\documentclass{amsart}

\usepackage{cute_font}

\usepackage[
  backend=biber,
    style=nature,
  ]{biblatex}
\author{Alex Burgin}
\address{Duke University, Durham, NC 27708, United States}
\email{alexander.burgin@duke.edu}

\author{Samuel Goldberg}
\address{University of Virginia, Charlottesville, VA 22904, United States}
\email{seg8st@virginia.edu}

\author{Tam\'as Keleti}
\address{Institute of Mathematics, E\"otv\"os Lor\'and University, 1117 Budapest, P\'azm\'any P. stny 1/c}
\email{tamas.keleti@gmail.com}
\thanks{The third author was supported by the Hungarian National Research, Development ad Innovation Office - NKFIH, 124749 and 129335.}

\author{Connor MacMahon}
\address{Michigan State University, East Lansing, MI 48824, United States}
\email{macmaho1@msu.edu}

\author{Xianzhi Wang}
\address{Middlebury College, Middlebury, VT 05753, United States}
\email{xianzhiw@middlebury.edu}

\title{Large sets avoiding infinite arithmetic / geometric progressions}

\begin{document}
\begin{abstract}
    We study some variants of the Erd\H{o}s similarity problem.
    We pose the question if every measurable subset of the real line with positive measure contains a similar copy of an infinite geometric progression.
    We construct a compact subset $E$ of the real line such that $0$ is a Lebesgue density point of $E$, but $E$ does not contain any (non-constant) infinite geometric progression.
    We give a sufficient density type condition that guarantees that a set contains an infinite geometric progression.
    
    By slightly improving a recent result of Bradford, Kohut and Mooroogen \cite{BKM}, we construct a closed set $F\subset[0,\infty)$ such that the measure of $F\cap[t,t+1]$ tends to $1$ at infinity but $F$ does not contain any infinite arithmetic progression.
    We also slightly improve a more general recent result by Kolountzakis and Papageorgiou \cite{KP} for more general sequences.
    
   We give a sufficient condition that guarantees that a given Cantor type set contains at least one infinite geometric progression with any quotient between $0$ and $1$. 
   This can be applied to most symmetric Cantor sets of positive measure.
\end{abstract}

\maketitle

\section{Introduction}


In 1974 P\'al Erd\H{o}s \cite{erdos-orig-citation} 
posed the question if for any infinite set $A\subset\R$
there exists a set $E\subset \mathbb{R}$ with positive Lebesgue measure such that $E$ contains no similar copies of $A$ (here by a \emph{similar copy} of a set $S\subset \R$ we mean a set of the form $aS+b=\{ax+b : x\in S\}$, where $a,b\in\R$ and $a\neq 0$). 
The statement that the answer is positive is called the \emph{Erd\H{o}s similarity conjecture}, although it is not clear if Erd\H{o}s himself expected an affirmative answer. 
There are a lot of partial results (see e.g. \cite{falconer-slowgrow,Sv, kolountzakis_1997, bourgainSumset, latestKolountzakis}), but none so far that conclusively answer the question.



It is easy to see that it is enough to check the conjecture for $A=\{a_1,a_2,\ldots\}$, where $(a_n)$ is a sequence converging to zero.
In 1984 Falconer \cite{falconer-slowgrow} showed the conjecture for all slowly converging sequences (in the sense that $a_{n+1}/a_n\rightarrow 1$). 
Surprisingly, there is no single exponentially quickly convergent sequence $(a_n)$ for which it is known if the Erd\H{o}s similarity conjecture holds. 
In particular it is not known if there exists a $q\in(0,1)$ such that every measurable set $E\subset\R$ with positive Lebesgue measure contains a similar copy of the sequence $(q^n)$.
Clearly, the Erd\H{o}s similarity conjecture would imply that such $q$ cannot exist. 
Here we pose the following possibly easier question:

\begin{question}\label{q:gptranslate}
Does every measurable set $E\subset\R$ of positive Lebesgue measure contain a sequence of the form $(a q^n + b)_{n=1}^\infty$ for some $a,b,q\in\R$, $a\neq 0$, $q\in(0,1)$? 
\end{question}

We do not know the answer to Question~\ref{q:gptranslate}.
A natural attempt to prove an affirmative answer could be taking a Lebesgue density point $b$ of $E$ and showing that for such a $b$ we can find $a\in\R\setminus\{0\}$ and $q\in(0,1)$ such that $(a q^n + b)_{n=1}^\infty\subset E$. 
One of the main results of this paper is the following result, which shows that this approach does not work.
\begin{theorem}\label{t:densitywithoutGP}
There exists a compact set $E\subset\R$ such that
$0$ is a Lebesgue density point of $E$ but
$E$ does not contain any (non-constant) infinite
geometric progression $(a q^n)_{n=1}^\infty$.
\end{theorem}



After seeing that a density point at zero does not guarantee that the set contains an infinite 
geometric progression, it is natural to seek for other density type conditions that do guarantee geometric progressions. 
In Proposition~\ref{prop:SuffIntCondition} we
give such a sufficient condition.
Then we show that there exist measurable sets of
positive Lebesgue measure for which this density type condition does not hold for any point.
Thus this approach cannot give an answer to Question~\ref{q:gptranslate} either.

In \cite{kolountzakis_1997} Kolountzakis proved that for any infinite set $A\subset\R$ there exists a measurable set $E\subset[0,1]$ of measure arbitrarily close to $1$ such that almost every similar copy of $A$ is not contained in $E$.
By a slight modification of the original proof of Kolountzakis \cite[Theorem 1]{kolountzakis_1997} we show that 
there exists a measurable set $E \subset [0,1]$ of measure arbitrarily close to $1$ such that almost every similar copy of almost every infinite geometric progression is not contained in $E$.

Note that $E$ contains an infinite 
geometric progression if and only if 
$\log(E\cap(0,\infty))$ or $\log((-E)\cap(0,\infty))$
contains an infinite arithmetic progression. 
Therefore the problem of finding a sufficient condition that guarantees that a set contains an infinite 
geometric progression is closely related to finding a sufficient condition that guarantees that a set contains an infinite arithmetic progression.
This latter problem has been studied by L.~Bradford, H.~Kohut and Y.~Mooroogen in a very recent paper \cite{BKM}.
They proved the following result.
Here and in the sequel $\lambda$ denotes Lebesgue measure.

\begin{theorem}[Bradford et al. \cite{BKM}]
\label{thm:Bradford}
For any $\varepsilon >0$
there exists $S \subset \mathbb{R}$ satisfying 
\begin{align*}
    \lambda(S\cap [t,t+1]) \geq 1-\varepsilon
\end{align*} for each real $t$ such that $S$ does not contain any infinite arithmetic progression.
\end{theorem}

We slightly improve the above result by proving the following.

\begin{theorem}\label{t:AP}
There exists a closed set $F\subset[0,\infty)$ satisfying 
\begin{equation}\label{e:limit1}
    \lim_{t\to\infty}\lambda(F\cap [t,t+1])=1
\end{equation}
such that $F$ does not contain any infinite arithmetic progression. 
\end{theorem}

The importance of this improvement is that in Lemma~\ref{lem:windowDensity} we prove that \eqref{e:limit1} holds if and only if $0$ is a density point of $\exp(-F)\cup -\exp(-F)$. 
Thus Theorem~\ref{t:AP} will easily imply Theorem~\ref{t:densitywithoutGP}. 

In an even more recent paper Kolountzakis and Papageorgiou \cite{KP} improved Theorem~\ref{thm:Bradford} in the following different way. 

\begin{theorem}[Koluntzakis-Papageorgiou \cite{KP}]
\label{t:KP} 
Let $A=\{a_1, a_2,\ldots\}\subset (0,\infty)$ such that
\begin{equation}\label{e:KPcondition}
    a_{n+1}-a_n\ge 1 \quad (\forall n\in\N)
    \qquad\text{and}\qquad
    \lim_{n\to\infty}\frac{\log a_n}{n}=0.
\end{equation}
Then for every $\varepsilon>0$ there exists a Lebesgue measurable set $E\subset\R$ such that
\begin{equation}\label{e:1minusepsinKP}
    \lambda(E\cap[t,t+1])\ge 1-\varepsilon 
    \qquad (\forall t\in\R),
\end{equation}
but $E$ does not contain any similar copy of $A$.
\end{theorem}

We slightly improve this result by showing Theorem~\ref{t:improvedKP}, which states that, similarly as above, in Theorem~\ref{t:KP} we can replace \eqref{e:1minusepsinKP} by the requirement
\begin{equation*}
    \lim_{t\to\infty}\lambda(E\cap [t,t+1])=1.   
\end{equation*}

Finally, in Section~\ref{s:Cantor} we study the following two closely related problems:
How can we guarantee that a given subset of $\R$ contains infinite arithmetic progressions with every positive common difference,
or infinite geometric progressions with every quotient between $0$ and $1$? 
As an application we show (Corollaries~\ref{cor:dyadicsymmetric_new} and \ref{cor:mafatCantor}) that under mild conditions all
symmetric Cantor sets with positive measure 
contain geometric progressions with any given quotient between $0$ and $1$.

To avoid trivial examples, by arithmetic progressions and geometric progressions we mean only non-constant sequences. 
Since any infinite geometric progression with negative ratio contains an infinite geometric progression with positive ratio, and since geometric progressions with ratio greater than $1$ are not contained in any bounded set, we consider only geometric progressions with ratio between $0$ and $1$.

\section{Construction of Large Sets without infinite arithmetic / geometric progressions}

In this section we prove Theorems~\ref{t:densitywithoutGP} and \ref{t:AP} and the improvement of Theorem~\ref{t:KP} mentioned in the Introduction. In order to prove Theorem~\ref{t:AP} we need Theorem~\ref{thm:Bradford}. 
To make the proofs of Theorems~\ref{t:densitywithoutGP} and \ref{t:AP} self-contained we give a new proof of Theorem~\ref{thm:Bradford}, which is considerably shorter than the original proof in \cite{BKM}.

First we prove two simple lemmas from which Theorem~\ref{thm:Bradford} will follow very easily.

\begin{lemma}
\label{lem:equidistribute}
Given $\varepsilon > 0$, we can find $E \subset [0,\infty)$ which contains no arithmetic progression of irrational common difference such that $\lambda(E \cap [t,t+1]) \geq 1- \varepsilon$ for 
every $t\ge 0$.
\end{lemma}
\begin{proof}
Fix $\varepsilon >0$. Let $E := \bigcup_{n=0}^\infty [n,n+1-\varepsilon]$. Clearly, $\lambda(E \cap [t,t+1]) \geq 1- \varepsilon$ for 
every $t\ge 0$.
Let $(b+n\Delta)_{n=0}^\infty$ be an arithmetic progression with $\Delta$ irrational. Since $\Delta$ is irrational, $(\{ b+n\Delta \})_{n=0}^\infty$ is dense in $[0,1)$, where $\{ \cdot \}$ denotes the fractional part. Hence, $E$ contains no arithmetic progression of irrational common difference $\Delta$. 
\end{proof}

\begin{lemma}
\label{Lem:countableElim}
Fix $\varepsilon >0$. 
Let $\Delta_1,\Delta_2,\dots$ be a countable collection of positive real numbers. Then there exists a measurable $E \subset [0,\infty)$ containing no arithmetic progression with common difference $\Delta_n$ for all $n$, and such that $\lambda(E \cap [t,t+1]) \geq 1-\varepsilon$ for every non-negative real $t$. 
\end{lemma}

\begin{proof}
Let $\phi:\mathbb{N} \rightarrow \mathbb{N} \times \mathbb{N}$ be onto and denote $\phi(n) = (\phi_1(n),\phi_2(n))$. 
Let $m_1,m_2,\ldots$ be a quickly increasing sequence of positive integers that we choose later. Let
\begin{equation*}
    J_n=((\phi_2(n)-1)\varepsilon/2, (\phi_2(n)+1)\varepsilon/2)
    \qquad\text{and}\qquad
    I_n=J_n+m_n \Delta_{\phi_1(n)}.
\end{equation*}
Since $I_n$ intersects every arithmetic progression $(b+k\Delta_{\phi_1(n)})_{k=0}^\infty$ for
every $b\in J_n$ and $\phi$ is onto, 
the set $E=[0,\infty)\setminus \cup_{n=1}^\infty I_n$ does not contain an infinite arithmetic progression with common difference $\Delta_n$ for all $n$.

By induction we can make the sequence $(m_n)$ grow fast enough to guarantee that $\inf I_n > 1+\sup I_{n-1}$ for every $n$.
This property clearly guarantees that 
$\lambda(E \cap [t,t+1]) \geq 1-\varepsilon$ for every non-negative real $t$. 
\end{proof}

\begin{proof}[Proof of Theorem~\ref{thm:Bradford}]
Fix $\varepsilon>0$. 
Let $E_1$ be the set we obtain by applying Lemma~\ref{lem:equidistribute} for $\varepsilon/4$.
Let $\Delta_1, \Delta_2, \ldots$ be an enumeration of $\Q\cap(0,\infty)$ and let $E_2$ be the set we obtain by applying Lemma~\ref{Lem:countableElim} to this sequence and $\varepsilon/4$.
Then it can be easily seen that $E=(E_1\cap E_2)\cup -(E_1\cap E_2)$ satisfies all the requirements.
\end{proof}

To get our improvements we need the following lemma.

\begin{restatable}{lemma}{thmGenCompactArg} 
\label{lem:Gen_Compact_Arg}
Let $\alpha_n$ be an increasing sequence of positive reals such that $\alpha_n\to\infty$. 
Suppose that for any $\varepsilon>0$ and $n\in\mathbb{N}$ 
there exists a measurable set $E(\varepsilon,n)\subset \mathbb{R}$
such that $E(\varepsilon,n)$ does not contain any similiar copy
of $\{\alpha_n, \alpha_{n+1},\ldots\}$ and 
\[
\lambda(E(\varepsilon,n)\cap[t,t+1])>1-\varepsilon 
\qquad (\forall t>0).
\]

Then there exists a closed set $F\subset\mathbb{R}$ that does
not contain any similiar copy
of $\{\alpha_n, \alpha_{n+1},\ldots\}$ for any $n$ and 
\[
\lim_{t\to\infty}\lambda(F\cap[t,t+1])=1.
\]
\end{restatable}

\begin{proof}
For $k\in\mathbb{N}$ let $E_k=\cap_{n=1}^\infty 
E(1/k, n)$.
Then $E_k$ does not contain any similar copy of $\{\alpha_n, \alpha_{n+1},\ldots\}$ for any $n$ and 
\[
\lambda(E_k\cap[t,t+1])>1-\frac1k 
\qquad (\forall t>0).
\]
By taking a closed subset we can suppose that $E_k$ is closed.
Let $G_k=\mathbb{R}\setminus E_k$.

For any set $X\subset \mathbb{R}$ let
\[
H_m(X)=\{(a,b)\in (0,\infty)\times\mathbb{R}\ :\ 
(\exists n> m)\ a\alpha_n + b \in X \}.
\]
Note that $H_m(X)$ is open if $X$ is open and by assumption
$H_m(G_k)=(0,\infty)\times\mathbb{R}$ for any $m$ and $k$.
Hence
$$
[1/k,k]\times[-k,k]\subset
\bigcup_{h>0} H_m(G_k \cap (-\infty,h)),
$$
so by compactness there exists an $h_{k,m} > 0$ such that
$$
[1/k,k]\times[-k,k]\subset H_m(G_k \cap (-\infty,h_{k,m})).
$$

Since for any $(a,b)\in[1/k,k]\times[-k,k]$ and $n> m$ we have
$a\alpha_n + b > \frac{\alpha_m}k -k $, we obtain that 
$$
[1/k,k]\times[-k,k] \subset 
H_m\left(G_k \cap \left(\frac{\alpha_m}k -k, h_{k,m}\right)\right).
$$
By induction, since $\alpha_n \to \infty$, we can choose a sequence of positive integers $m_k$ such that
$$
\frac{\alpha_{m_k}}{k}-k > h_{k-1,m_{k-1}}+1.
$$

Finally, let
$$
G=\bigcup_{k=1}^\infty G_k\cap \left(\frac{\alpha_{m_k}}k -k, h_{k,m_k}\right)
$$
and take $F=[0,\infty)\setminus G$. 
It is easy to check that $F$ has all of the required properties.
\end{proof}

\begin{proof}[Proof of Theorem~\ref{t:AP}]
Let $\alpha_n=n$. Then for each $n$ the similar copies of $\{\alpha_n,\alpha_{n+1},\ldots\}$ are exactly the infinite arithmetic progressions. Thus, by Lemma~\ref{lem:Gen_Compact_Arg}, Theorem~\ref{thm:Bradford} implies Theorem~\ref{t:AP}.
\end{proof}

Lemma~\ref{lem:Gen_Compact_Arg} can be also applied to strengthen slightly Theorem~\ref{t:KP} of Kolountzakis and Papageorgiou.

\begin{theorem}
\label{t:improvedKP} 
Let $A=\{a_1, a_2,\ldots\}\subset (0,\infty)$ such that the following hold:
\begin{equation}\label{e:KP2}
    a_{n+1}-a_n\ge 1 \quad (\forall n\in\N),
\end{equation}
\begin{equation}\label{e:KP3}
    \lim_{n\to\infty}\frac{\log a_n}{n}=0.
\end{equation}
Then there exists a Lebesgue measurable set $E\subset\R$ such that
\begin{equation}\label{e:limitinKP}
\lim_{t\to\infty}\lambda(E\cap[t,t+1])=1.
\end{equation}
but $E$ does not contain any similar copy of $A$.
\end{theorem}

\begin{proof}
To check that Lemma \ref{lem:Gen_Compact_Arg} can be applied to Theorem~\ref{t:KP} to obtain Theorem~\ref{t:improvedKP}, it suffices to show that for any sequence $(a_n)$ satisfying \eqref{e:KP2} and \eqref{e:KP3}, each of its tails $(a_n)_{n>k}$ also satisfies \eqref{e:KP2} and \eqref{e:KP3}. For \eqref{e:KP2} this is clear. 
For \eqref{e:KP3}, fix $k \in \mathbb{N}$ and let $b_n = a_{n+k}$ for all $n \in \mathbb{N}$. Since $\{a_n\}$ satisfies \eqref{e:KP3}, we have the following:
\begin{align*}
    \lim_{n \rightarrow \infty} \frac{\log b_n}{n} = \lim_{n \rightarrow \infty} \frac{\log(a_{n+k})/(n+k)}{n/(n+k)} =  \frac{\lim_{n \rightarrow \infty}\log(a_{n+k})/(n+k)}{\lim_{n \rightarrow \infty }n/(n+k)} = 0
\end{align*}
Hence, the sequence $(b_n)$ satisfies \eqref{e:KP3}, which completes the proof.
\end{proof}

In order to show that Theorem~\ref{t:AP} implies Theorem~\ref{t:densitywithoutGP} we need the following Lemma.
\begin{lemma}
\label{lem:windowDensity}
The following three conditions are equivalent for any measurable $E \subset (0,\infty)$:
\begin{align*}
    &\text{(i)}\quad \lim_{t \rightarrow 0^+} \frac{\lambda(E \cap [0,t])}{t} = 1,\\
    &\text{(ii)}\quad \lim_{\mu \rightarrow \infty}\frac{\lambda(-\log E \cap [\mu, \mu+\nu])}{\nu} = 1 \quad \text{for all fixed} \quad \nu>0,\\ 
   &\text{(iii)}\quad \lim_{\mu \rightarrow \infty}{\lambda(-\log E \cap [\mu, \mu+1])} = 1.\end{align*}
\end{lemma}

\begin{proof}

Let $F=(0,\infty)\setminus E$.
Writing $\lambda(E\cap [0,t])/t = 1 - \lambda(F\cap [0,t])/t = 1 - \frac{1}{t} \int_{0}^t \mathds{1}_F(r) d\lambda(r)$ and applying a change of variable $x = - \log r$, condition $(i)$ is equivalent to
\begin{align*}
 (i')\qquad\qquad   \lim_{\mu \rightarrow \infty} e^{\mu} \int_{\mu}^{\infty} \mathds{1}_{-\log F}(x) e^{-x} d \lambda(x) = 0,
\end{align*}
where $t = e^{-\mu}$. 
We can clearly rewrite 
(ii) as
\begin{align*}
(ii')\qquad (\forall\nu>0)\quad    \lim_{\mu \rightarrow \infty} \int_{\mu}^{\mu +\nu} \mathds{1}_{-\log F}(x) d\lambda(x)=0,
\end{align*}
and (iii) as
\begin{align*}
(iii')\qquad \qquad \qquad  \lim_{\mu \rightarrow \infty} \int_{\mu}^{\mu +1} \mathds{1}_{-\log F}(x) d\lambda(x)=0,
\end{align*}
therefore it is enough to prove $(i') \Longleftrightarrow (ii')\Longleftrightarrow (iii')$.

First we show $(i') \implies (ii')$.
We have the following for all $\mu,\nu > 0$ by the non-negativity of the function $\mathds{1}_{-\log F}(x) e^{-x} $ and the monotonicity of the integral:
\begin{align*}
    e^{\mu+\nu
    }\int_{\mu}^\infty \mathds{1}_{-\log F}(x) e^{-x} d\lambda(x) \geq  \int_{\mu}^{\mu + \nu}\mathds{1}_{-\log F}(x) e^{\mu +\nu 
    -x}d\lambda(x) \geq \int_{\mu}^{\mu + \nu}\mathds{1}_{-\log F}(x)d\lambda(x) \geq 0.
\end{align*}

 Now, letting $\mu \rightarrow \infty$ and applying the Squeeze Theorem with condition $(i')$, we obtain $(ii')$. 
 
Since $(ii') \implies (iii')$ is clear it remains to  show $(iii') \implies (i')$. 
For any positive integer $k$ we have
\begin{align*}
   e^{\mu} \int_{\mu}^{\infty}\mathds{1}_{-\log F}(x) e^{-x} d \lambda(x) 
   &\le \sum_{j=1}^k\int_{\mu+j-1}^{\mu+j} \mathds{1}_{- \log F}(x) e^{\mu -x } d\lambda(x) + \int_{\mu+k}^{\infty} e^{\mu -x } d\lambda(x)\\
   &\leq k\cdot \int_{\mu}^{\mu+1} \mathds{1}_{-\log F} (x) d \lambda(x) + e^{-k}.
\end{align*} 
Letting $\mu \rightarrow \infty$ and using $(iii')$ we obtain that 
the limit of the left-hand side is at most $e^{-k}$
for all $k$, so 
the limit is zero as desired.
\end{proof}




\begin{proof}[Proof of Theorem~\ref{t:densitywithoutGP}]
Let $F$ be a set provided by Theorem \ref{t:AP}, let $E=\exp(-F)$ and $C=E \cup -E\cup \{0\}$.
Since $F\subset[0,\infty)$ is closed, $C$ is compact.
Since $F$ does not contain any infinite arithmetic progression, $E$ does not contain any infinite geometric progression and neither does $C$.
Since $F=-\log E$, \eqref{e:limit1} of Theorem~\ref{t:AP} implies that (iii) of Lemma~\ref{lem:windowDensity} holds.
Thus (i) of Lemma~\ref{lem:windowDensity} implies that $0$ is indeed a Lebesgue density point of $E$.
\end{proof}

\section{A sufficient density condition that guarantees infinite geometric progression}

We give a sufficient condition for a measurable set to contain an infinite geometric progression.

\begin{restatable}{proposition}{propIntCondition}
\label{prop:SuffIntCondition}
 Suppose that $E\subset\R$ is a measurable set and for some fixed $r_0>0$ we have
\begin{equation}\label{e:convdensitysum}
   \sum _{n=1} ^{\infty } \int _{E^c \cap [0,r_0 ^n]} \frac{1}{n} x^{\frac{1}{n}-1} \, d\lambda (x) < \infty,
 \end{equation}
where $E^c$ denotes the complement of $E$.
Then there exists $a>0$ and $r \in (0,1)$ such that $(ar^n)_{n \in \mathbb{N }} \subset E$.
\end{restatable}


\begin{proof}
We proceed by proving the contrapositive,
so we suppose that there exists no $a>0$ and $r\in(0,1)$ such that $(ar^n)_{n \in \mathbb{N }}\subset E$.
Then for each $r \in (0,1)$, there exists a strictly increasing sequence of natural numbers $(n_m) _{m \in \mathbb{N}}$ dependent on $r$ such that $(r^{n_m})_{m \in \mathbb{N}} \subset E^c $. Now, fix $m \in \mathbb{N} \setminus \{ 0 \}$. 
Let $G_n(x)=x^\frac1n$.
By $\sigma$-subadditivity
and by applying a change of variables $r=G_n(x)=x^\frac1n$
we obtain
\begin{align*}
r_0 &= \lambda ([0,r_0] ) = \lambda ( \{ r \in [0,r_0 ] : r^n \in E^c \text{ for some } n \geq m \} )\\
&\leq \sum _{n=m } ^{\infty } \lambda ( \{ r \in [0,r_0] : r^n \in E^c \}
= \sum _{n=m } ^{\infty } \int _{G_n(E^c) \cap [0,r_0] } \, d \lambda (r) 
= \sum _{n=m } ^{\infty } \int _{E^c \cap [0,r_0 ^{n}] } \frac{1}{n} x^{\frac{1}{n} -1 } \, d \lambda (x).
\end{align*}
Since we have bounded the tail of the non-negative series below by a positive constant, it diverges to infinity.
\end{proof}

\begin{remark}
Since $\int_0^{r_0^n} \frac1n x^{\frac1n-1}\ dx=r_0$,
the series of \eqref{e:convdensitysum} can be viewed as $r_0$ times a series of increasingly narrower weighted densities of $E^c$ at 0. Thus, the sufficient condition is that these densities of $E^c$ converge to $0$ quickly enough 
to make this series converge. From Theorem \ref{t:densitywithoutGP}, we know there exists a set with $0$ as a density point which does not satisfy the sufficient condition. Yet, given measurable $E \subset \mathbb{R}$, we might hope that by translating the condition in the above proposition, we could find a viable translate for a copy of a geometric progression. However, there exist sets of large measure which do not satisfy this condition at any point with density 1, for we can make the density converge as slowly as we like:
indeed, it is well known that there is no bound for the rate of convergence at a Lebesgue density point, see for example \cite{Taylor1959}. For our purposes, this fact manifests itself in the following manner.

\begin{lemma}
\label{lem:monster}
Fix $\varepsilon >0$ and let $f:(0,1) \rightarrow \mathbb{R}^+$ be a continuous, monotone 
non-decreasing
function such that $\lim_{x \rightarrow 0^+} f(x) = 0$. Then there exists an open $G \subset (0,1)$ with measure not exceeding $\varepsilon$, so that for any $p \in [0,1)$ and $t\in(0,1-p)$ we have
\begin{equation*}
      \frac{\lambda(G \cap [p,p+t])}{t} \geq f(t).
 \end{equation*}
\end{lemma}

\begin{proof}
Fix $\varepsilon>0$. Choose a sequence $1=\delta_1>\delta_2>\ldots$ such that $f(\delta_n)<\varepsilon/2^{n+1}$ for every $n$.
For each $n$, choose $M_n \gg 0$ satisfying
\begin{align*}
    \frac{(\lfloor M_nx \rfloor - 2)}{M_nx} > \frac{1}{2}
\end{align*}
for all $x \in [\delta_{n+1},\delta_n)$. Set $G = \bigcup_{n=1}^\infty \bigcup_{k=0}^{M_n-1} (\frac{k}{M_n},\frac{1}{M_n}(k+\frac{\varepsilon}{2^n}))$. By subadditivity, the measure of $G$ does not exceed $\varepsilon$. Now, fix $p \in [0,1)$ and 
$t\in(0,1-p)$.
Since 
$0<t<1=\delta_1$,
there exists a natural number $n_0$ so that $t \in [\delta_{n_0+1},\delta_{n_0})$. Clearly, $(\frac{k}{M_{n_0}},\frac{k+1}{M_{n_0}}) \subset [p,p+t]$ for at least $\lfloor M_{n_0}t \rfloor - 2$ integers $0 \leq k \leq M_{n_0}-1$. Hence, we have the following:
\begin{align*}
    \frac{\lambda(G \cap [p,p+t])}{t} \geq \frac{1}{t}\lambda \bigg[ \bigcup_{k=0}^{M_{n_0}-1} \bigg( \frac{k}{M_{n_0}}, \frac{1}{M_{n_0}}\bigg(k + \frac{\varepsilon}{2^{n_0}}\bigg) \bigg) \cap [p,p+t] \bigg] \\ \geq \frac{(\lfloor M_{n_0}t \rfloor -2)\varepsilon }{M_{n_0}t2^{n_0}} > \frac{\varepsilon}{2^{n_0+1}} 
    \ge f(\delta_{n_0})\geq f(t).
\end{align*}
\end{proof}


Applying the above with $f$ as follows
\begin{align*}
    f(x) = \begin{cases} \frac{1}{\log( \log(1/x))} & x < e^{-e} \\ 1 & x \geq e^{-e} \end{cases}.
\end{align*}
one can construct a closed subset $E$ of $\mathbb{R}$ with cofinite measure such that $E-p$ does not satisfy the condition in Proposition \ref{prop:SuffIntCondition} for all $p \in \mathbb{R}$. 
\end{remark}

\subsection{A Remark on a Result of Kolountzakis}

In \cite{kolountzakis_1997} Kolountzakis proved that for any infinite set $A\subset\R$ there exists a measurable set $E\subset[0,1]$ of measure arbitrarily close to $1$ such that almost every similar copy of $A$ is not contained in $E$; that is,
\begin{equation*}
    \lambda(\{(a,b)\in\R^2 \ :\ aA+b \subset E\}) =0.
\end{equation*}

Since in this paper our goal (see Question~\ref{q:gptranslate}) is to find at least one similar copy of at least one infinite geometric progression it is natural to ask if there exists a set $E$ as in the above theorem that works for $A=\{r^n : n\in\N\}$ simultaneously for every $r\in(0,1)$ or at least for almost every $r\in(0,1)$.
It turns out that for almost every $r\in[0,1]$ this can be obtained by a slight modification of the original proof of Kolountzakis \cite[Theorem 1]{kolountzakis_1997}.

\begin{proposition}
\label{p:modifiedK}
There exists a measurable set $E \subset [0,1]$ of measure arbitrarily close to $1$ such that almost every similar copy of almost every infinite geometric progression is not contained in $E$; that is,
\[
\lambda ( \{ (r,a,b) : (ar^n+b)_{n \in \mathbb{N}} \subset E \} )= 0.
\]
\end{proposition}

\begin{proof}
The proof is nearly the same as in \cite[Theorem 1]{kolountzakis_1997}, so we just take note of the modifications. First, notice that, the same argument that shows that we can suppose that the scaling factor $a$ is in some fixed interval $[\alpha,\beta]$ for some $0<\alpha<\beta<\infty$,  also shows that it suffices to consider only $r \in [ r_0 , r_1 ]$ where $0<r_0 < r_1 < 1$.
Choose $m_j$ such that $1/m_j < \frac12\alpha r_0^j(1-r_1).$ Then $1/m_j$ is smaller than half of the minimum gap of the numbers $\alpha r^0,\ldots, \alpha r^{j-1}$ any $r\in[r_0,r_1]$ as in the original proof.
For the final step, observe from the last computation in the proof that we have
\[
\mathbb{E} \lambda ( \{ (a,b,r) : (a+br^n) _{n \in \mathbb{N}} \subset E \} ) 
= \int _{r_0} ^{r_1} \mathbb{E} \lambda ( \{ (a,b) : (a+br^n) _{n \in \mathbb{N}} \subset E \} ) d \lambda (r) =0,
\]
and obtain our desired set with positive probability as in the original proof.
\end{proof}

\section{A sufficient condition for Cantor-type sets to contain scaled copies of all convergent infinite geometric progressions.}
\label{s:Cantor}

First we give a sufficient condition that guarantees arithmetic progressions with any positive common difference.

\begin{lemma}\label{l:BCforAP}
If $A\subset\R$ is a measurable set of finite Lebesgue measure
then for every $\Delta>0$ the set $(0,\infty)\setminus A$
contains infinite arithmetic progression with difference $\Delta$.
\end{lemma}

\begin{proof}
Without loss of generality we can suppose that $\Delta=1$.
Let 
$$
B_n=A\cap(n,n+1]-n\subset(0,1].
$$
Then 
\begin{equation*}
    \sum_{n=0}^\infty \lambda(B_n)=
    \sum_{n=0}^\infty \lambda(A\cap(n,n+1])=
    \lambda(A\cap(0,\infty))<\infty.
\end{equation*}
Thus, by the Borel-Cantelli lemma, almost every $a\in(0,1]$
is contained only in finitely many $B_n$.
On the other hand, if 
$a\in(0,1]$ is contained only in finitely many $B_n$
then for large enough $N$ the arithmetic progression 
$a+N, a+N+1, a+N+2,\ldots$ is contained in $(0,\infty)\setminus A$.
\end{proof}

Now we can give a condition that guarantees an infinite converging geometric progression for every possible quotient.
\begin{lemma}\label{l:or}
Let $I_k=(u_k,v_k)\subset (0,\infty)$
$(k=1,2,\ldots)$ be disjoint open intervals and
$G=\cup_{k=1}^\infty I_k$. If
\begin{equation}\label{e:convsum}
\sum_{k=1}^\infty 
\log v_k - \log u_k < \infty
\qquad\text{or}\qquad
\sum_{k=1}^\infty
\frac{v_k-u_k}{v_k}<\infty    
\end{equation}
then for every $q\in(0,1)$ the set 
$E=(0,\infty)\setminus G$
contains a similar copy of $\bigc{q^n}_{n \geq 0}$ converging to $0$.
\end{lemma}

\begin{proof}
We claim that the second inequality of \eqref{e:convsum} implies the first one, so it is enough to prove the lemma when the first inequality of \eqref{e:convsum} holds. 
Indeed, the second inequality implies that $u_k/v_k\to 1$,
so $v_k\le 2u_k$ for large enough $n$.
Thus, using that for any $0<u<v\le 2u$ we have
\begin{equation*}
    \log v - \log u = \log\left(1+\frac{v-u}{u}\right)\le\frac{v-u}{u}
    \le 2\frac{v-u}{v},
\end{equation*}
we obtain the claim.

Let $A=-\log(G)$. Then
\begin{equation*}
    \lambda(A)=\sum_{k=1}^\infty (\log v_k - \log u_k)<\infty.
\end{equation*}
By Lemma~\ref{l:BCforAP} this implies that $(0,\infty)\setminus A$ contains
infinite arithmetic progression with any given difference $\Delta>0$.
This clearly implies that 
$E=(0,\infty)\setminus G$
contains an infinite geometric progression with any quotient $q\in(0,1)$.
\end{proof}

\begin{definition}\label{d:dsCantor}
Fix a sequence $t_0,t_1,\ldots$ such that 
\begin{equation}\label{e:symmetricCantor}
  \sum_{k=0}^\infty 2^k t_k \le 1.  
\end{equation}
We define the sets $E_0\supset E_1 \supset\ldots$ by induction such that each $E_k$ is the union of $2^k$ closed intervals of equal length.
Let $E_0=[0,1]$. 
If $E_k$ is defined then $E_{k+1}$ is obtained by removing
an open interval of length $t_k$ from the middle of each closed
interval of $E_k$. 
Note that condition \eqref{e:symmetricCantor} guarantees that the length of the intervals of $E_k$ is greater than $t_k$.  
We call the set $E=\cap_{k=0}^\infty E_k$ a \emph{symmetric Cantor set} with parameters $(t_k)$. 
\end{definition}

\begin{corollary}\label{cor:dyadicsymmetric_new}
If 
\begin{equation}\label{e:convforCor}
\sum_{k=0}^\infty 2^k t_k < 1
\qquad\text{and}\qquad
 \sum_{k=0}^\infty 2^k \cdot t_k \cdot k < \infty  
\end{equation}
then the above defined symmetric Cantor set $E=\cap_{k=0}^\infty E_k$ with
parameters $(t_k)$ contains an infinite 
geometric progression with every quotient $q\in(0,1)$.
%
\end{corollary}

\begin{proof}

For each $k$ let $I_{k,j}=(u_{k,j},v_{k,j})$ 
with $u_{k,1}<\ldots<u_{k,2^k}$
be the open intervals of length $t_k$ we removed from the middle of the closed intervals of $E_k$. 
Thus $v_{k,j}-u_{k,j}=t_k$ for each $k$ and $j$ and we have
$E=[0,1]\setminus\cup_{k,j} I_{k,j}$.

Let $m=\lambda(E)$. Note that the first inequality of \eqref{e:convforCor} gives that $m>0$.
By definition, the closed intervals of $E_k$ have length 
$\lambda(E_k)/2^k\ge \lambda(E)/2^k=m/2^k$.
This implies that $v_{k,j}\ge (j-1/2)\cdot m/2^k$ for each $k,j$.

Therefore
\begin{equation*}
    \sum_{k=0}^\infty \sum_{j=1}^{2^k} \frac{v_{k,j}-u_{k,j}}{v_{k,j}} \le
    \sum_{k=0}^\infty \sum_{j=1}^{2^k} \frac{t_k}{(j-1/2)\cdot m/2^k} =
    \frac1m\cdot\sum_{k=0}^\infty 2^k\cdot t_k\cdot\sum_{j=1}^{2^k} \frac{1}{(j-1/2)}.
\end{equation*}
Since the inner sum of the right-handside is at most 
$C\log(2^k)=Ck\log2$ for some constant $C$,
the second inequality of \eqref{e:convforCor} implies
that Lemma~\ref{l:or} can be applied to complete the proof.
\end{proof}

\begin{definition}
We call a symmetric Cantor set $E\subset [0,1]$ a \emph{middle-$a$ fat Cantor set} if it is a symmetric Cantor set (see Definition~\ref{d:dsCantor}) with parameters $t_k=a^{k+1}$ for some $0<a<1/3$. Note that we necessarily need $a<1/3$ to preserve positive measure of our resultant set.
\end{definition}

\begin{corollary}\label{cor:mafatCantor}
All middle-$a$ fat Cantor sets contain an infinite 
geometric progression with every quotient $q\in(0,1)$.
%
\end{corollary}

\begin{proof}
It is easy to see that \eqref{e:convforCor} holds for
$t_k=a^{k+1}$ if $a\in(0,1/3)$, so Corollary~\ref{cor:dyadicsymmetric_new} can be applied.
\end{proof}

\section*{Acknowledgements}
The authors would like to thank the
Budapest Semesters in Mathematics program for providing the framework under
which this research was conducted.

\printbibliography
\end{document}